\documentclass[12pt]{article}
\usepackage{ms}
\usepackage[usenames,dvipsnames]{xcolor}
\usepackage{xspace}
\usepackage{perpage}
\MakePerPage{footnote}
\usepackage[hang,flushmargin]{footmisc}

\newcommand{\arhan}{Arhangel'ski\u{\i}\xspace}

\newcommand{\frech}{Fr\'echet-Urysohn\xspace}
\newcommand{\groth}{Gro\-then\-dieck\xspace}
\newcommand{\lind}{Lindel\"of\xspace}

\newcommand{\strl}{\text{Str}(L)}
\newcommand{\modt}{\text{Mod}(T)}
\newcommand{\udlc}{\mathrm{DULC}}

\newcommand\blfootnote[1]{%
  \begingroup
  \renewcommand\thefootnote{}\footnote{#1}%
  \addtocounter{footnote}{-1}%
  \endgroup
}

\renewcommand*{\thefootnote}{\fnsymbol{footnote}}

\title{On the Undefinability of Pathological Banach Spaces}
\author{Clovis Hamel${}^{1,2}$ \ and Franklin D. Tall${}^{2}$}
\date{\today}

\begin{document}
\maketitle
\blfootnote{${}^1$ The first author was an NSERC Vanier Scholar when most of this research was carried out.} \blfootnote{${}^2$ Research supported by NSERC Grants RGPIN-2016-06319 and RGPIN-2023-03420.}

\begin{abstract}
    Motivated by Tsirelson's implicitly defined pathological Banach space, T. Gowers asked whether explicitly defined  Banach spaces must include either $c_0$ or some $\ell^p$. J. Iovino and P. Casazza gave an affirmative answer for first-order continuous logic. We greatly extend their work to logics with much weaker requirements than compactness on their type spaces. Noteworthy is our extensive use of the topology of function spaces ($C_p$-theory) as developed by Arhangel'skii, and our use of double limit conditions studied by H. König and N. Kuhn.
\end{abstract}
\section{Introduction}
Classical Banach spaces all include either a copy of $c_0$ or a copy of some $\ell^p$. Call such a Banach space ``\emph{nice}''. The negation of ``nice'' is \emph{pathological}. Tsirelson \cite{Tsirelson1974} constructed by a novel method a pathological Banach space. Gowers \cite{Gowers1995} asked whether all definable infinite dimensional Banach spaces are nice. The question of course demands a definition of ``definable''. Logicians know what that means; it requires specifying a logic---its language, vocabulary, and, in general, its syntax. It also requires specifying its semantics---what is the set of objects to be quantified over, and which formulas are true. Assuming these have been done, an object is \emph{definable} if there is a formula defining it, i.e., $B = \{x : \phi(x)\}$.

The first attempt at answering Gowers' question was due to P.~Casazza and J.~N.~Iovino \cite{Casazza}, for a class of logics with enough expressive power to talk about Banach spaces. These logics were somewhat unwieldy, but were basically \emph{real-valued continuous logic}, where truth values are taken in $[0, 1]$ and there is enough expressive power to do elementary analysis. The usual $2$-valued logic is a special case of this, so their results apply to that as well. They made essential use of the compactness of their logic, i.e., if every finite subset of a set of formulas is satisfiable, then the whole set is satisfiable\footnote{\cite{Casazza} claims to just use the \emph{countable} compactness of their logics but there is a gap in their argument. In \cite{CDI}, the authors only claim their result for compact logics. We will prove the countably compact case in Section 6.}. Augmented by E.~Due\~nez, they have now developed a more elegant version \cite{CDI}. Both versions prove the undefinability of Tsirelson's space, but beg the question of whether there might be some other reasonable logic in which it could be defined. We recognized the essential topological nature of some of the ideas of \cite{Casazza} and have thus been able to vastly increase the scope of such undefinability results in \cite{Hamel} and \cite{Hamelb}. The first, \cite{Hamel}, deals with logics that have \lind first countable \emph{type spaces} (defined below). The second, \cite{Hamelb}, generalizes to logics that have \lind $\Sigma$ countably tight (defined below) type spaces. Here we generalize the previous four papers and require only a common generalization of compactness and countable tightness.

The papers \cite{Hamel} and \cite{Hamelb} are rather lengthy, since their aims are to introduce model theory to $C_p$-theorists, and $C_p$-theory to model theorists, respectively. ($C_p$-theory is that branch of general topology that deals with topological spaces composed of continuous real-valued functions with the pointwise topology.) Here, in contrast, we feel free to refer to the literature and not prove results that can easily be found elsewhere. In particular, for general topology we refer to \cite{Engelking1989}, for $C_p$-theory to \cite{ArhangelskiiFunction}, and for model theory to \cite{TentZiegler}. On the other hand, we hope we have put in sufficiently many definitions so that an analyst can follow our proofs.

Readers of \cite{Casazza}, \cite{Hamel}, and \cite{Hamelb} have found the formalism of ``language for pairs of structures'' unwieldy, so the authors of \cite{CDI} reverted to more traditional model theory. However, this came at a cost; it is unclear if their results can be generalized to logics other than compact ones --- see discussion in \cite{Hamelb} and below. We therefore have retained the formalism of \cite{Casazza}, which we don't consider all that bad. After all, just as we shall consider two Banach norms in one model, \emph{bitopologists} consider two metrics on one topological space.

In addition to the expository differences between this paper and \cite{Hamel} and \cite{Hamelb}, the important difference is that we greatly weaken the hypotheses on the spaces of types of the logics considered, getting rid of the \lind $\Sigma$ condition in \cite{Hamelb} and finding a natural weakening of both compactness and countable tightness that suffices to give us undefinability. This is not, however, the final word. Obviously, Tsirelson's space is definable in a logic in which there is a formula $\phi$ which is defined to hold for $x$ if and only if $x$ is Tsirelson's space. Can we claim that Tsirelson's space is not definable in any ``reasonable'' logic? We believe that is the case, but it is unclear how to define ``reasonable''. The space, after all, is also definable in the logic of first-order set theory. We believe the intention of Gowers' problem is to just consider logics that are appropriate for formalizing the analysis that analysts do. So $\varepsilon$'s and $\delta$'s, $<$, etc., but not the membership relation, which makes the problem (unsatisfactorily) go away if one assumes the usual axioms of set theory, including the Axiom of Choice. We succeed here in putting quite weak conditions on the logics considered in order to prove undefinability, but don't claim we have produced the final answer to Gowers' rather vague question.

\section{Model Theory}
We provide a brief introduction without proofs to the essential model theory we need. We first set out some elementary notions that will be used throughout. For a more detailed exposition, see \cite{TentZiegler}. A \emph{language} is a set of constants, function symbols and relation symbols, e.g. $\{\in\}$ is the language of set theory, where $\in$ is a binary relation, and $\{\bar{e}, *, ^{-1}\}$ is the language of group theory, where $\bar{e}$ is a constant, $*$ is a binary relation symbol and $^{-1}$ is an unary function symbol.
 \emph{First-order logic} is the usual logic used in mathematics: given a language $L$, 
 formulas are built recursively as certain finite strings of symbols built using the members of $L$, parentheses, variables, and the logical connectives
 $\wedge, \vee, \to, \leftrightarrow, \neg$ and quantifiers $\exists$, $\forall$.
 For the rest of this section let $L$ be a fixed language.

An $L$-structure $\modelm$ is a set $M$, called the \emph{universe} of $\modelm$ together with an \emph{interpretation} of elements of $L$, i.e., a function which assigns an element of $M$ to each constant symbol, a function from $M^n$ to $M$ to each $n$-ary function symbol, and a subset of $M^n$ to each $n$-ary relation symbol. Given an $L$-formula $\varphi$, we denote by $\varphi^{\modelm}$ the interpretation of $\varphi$ in $\modelm$.
A \emph{substructure} $\modeln$ of $\modelm$ is just an $L$-structure, the universe of which is $N$, a subset of $M$, containing all the same interpretations of symbols from $\modelm$. For instance, in the language of group theory, the formula $(\forall x)(\exists y)(x*y=\bar{e})$ is interpreted in the additive group $\ints$ as ``for every element $x$ of $\ints$, there is a $y$ in $\ints$ such that $x+y = 0$.''
 
A variable is said to be \emph{free} in a formula if it is not under the scope of any quantifier; a \emph{sentence} is a formula without free variables. A \emph{theory} is a set of sentences and their logical consequences; a theory is \emph{satisfiable} (or, \emph{has a model}) if $\phi^{\frak{M}}$ is true for every $\phi$ in $T$. A maximal satisfiable theory is said to be \emph{complete}. (Note \emph{consistency} is often defined proof-theoretically: not every sentence can be proved. The \emph{Completeness Theorem} for first-order logic asserts that consistency is equivalent to satisfiability). For example, the axioms and theorems of group theory constitute a theory. If $T$ is a theory, the \emph{satisfaction} relation $\modelm\models{T}$ is defined recursively to mean $\modelm$ is a model of $T$, i.e. $T$ is true in $\modelm$. For example, any group satisfies the axioms of group theory. Another example, 
$(\reals, +)\models(\forall x)(\forall y)(x*y=y*x)$.
Two models $\modelm$ and $\modeln$ are \emph{elementarily equivalent} if for any $L$-sentence $\varphi$ we have $\modelm\models{\varphi}$ if and only if 
$\modeln\models{\varphi}$; this is denoted
by $\modelm\equiv{\modeln}$.

If $\modelm$ is an $L$-structure and $A\subseteq{M}$, we can extend the language $L$ to $L(A)$ by adding a constant symbol 
$\dot{a}$ for each $a\in{A}$. In this case, an $L(A)$-formula is simply an $L$-formula with parameters from $A$.

Given a language $L$, two interesting topological spaces arise: the space of types and the space of $L$-structures. We shall now introduce them and state that the \emph{Compactness Theorem} for first-order logic is equivalent to either of those spaces being compact.
\begin{defi}\label{classictype}
Let $\modelm$ be an $L$-structure and $A\subseteq{M}$. A \emph{complete $n$-type} over $A$ in $\modelm$ is a maximal satisfiable set of $L(A)$-formulas in the variables $x_1, \dots, x_n$;
i.e., a maximal set $p(x_1, \dots, x_n)$ of formulas
$\varphi(x_1,\dots,x_n)$ such that there are
$a_1,\dots,a_n\in{M}$ such that for every $\varphi\in{p}(x_1, \ldots, x_n)$ we have $\modelm\models{\varphi(a_1,\dots,a_n)}$. The set of all complete $n$-types is denoted by  $\s^n(A)$.
\end{defi}

Given a formula $\varphi=\varphi(x_1,\dots,x_n)$, we define
$$[\varphi]=\{p\in{\s^n(A)} \colon \varphi\in{p}\}$$ where the superscript is omitted if it is $1$.

Note that all sets of the form $[\varphi]$ form a basis for a topology on $\s^n$. Moreover, since the negation: $\neg$ is available in first-order logic, each $[\varphi]$ is clopen. We will soon be dealing with other logics in which negation is not at our disposal and the $[\varphi]$'s will only constitute a basis for the closed sets. Most of the time, it is enough to study $1$-types, and then the same proofs can be carried out for $n$-types.

\begin{defi}\label{def:eqclass}
Let $\strl$ be the set of all equivalence classes under $\equiv$ of $L$-structures.
For each theory $T$ let 
$[T]=\{\widetilde{\modelm}\in{\strl} \colon \modelm\models{T}\}$ where $\widetilde{\modelm}$ is the equivalence class of $\modelm$. All the sets of the form $[T]$ constitute a basis
for the closed sets of the topology on $\strl$ known as the \emph{space of \mbox{$L$
-structures}}. We write $[\varphi]$ instead of $[\{\varphi\}]$.
\end{defi}

\begin{rmk*}
Considering equivalence classes under $\equiv$ instead of models is a way to avoid foundational issues: the upward L\"owenheim-Skolem theorem states that if a theory has a model of cardinality $\kappa$ then it has a model of cardinality $\lambda$ for every $\lambda\geq{\kappa}$. Consequently, dropping the modulo $\equiv$ would require us to deal with the proper class $\modt = \st{\modelm}{\modelm\models T}$. Considering equivalence classes is equivalent to simply considering the set of all satisfiable theories, which is a set of cardinality $\leq{2^{|L|}}$. A topological consequence of mod-ing out by $\equiv$ is that we guarantee that the topology on $\strl$ is Hausdorff. This is not necessarily the case otherwise, since there could be two different elementarily equivalent structures $\modelm$ and $\modeln$ which are not topologically distinguishable.
\end{rmk*}

Now we state a result which is a cornerstone of first-order logic:

\begin{thm}[The Compactness Theorem]
Let $T$ be an $L$-theory in first-order logic. If $T$ is finitely satisfiable, i.e., for any $\Delta\in{[T]^{<\omega}}$ there is an $L$-structure $\modelm_{\Delta}$ satisfying $\modelm_{\Delta}\models{T}$, then $T$ is satisfiable.
\end{thm}

\begin{proof}
  See \cite{TentZiegler}.
\end{proof}

The relationship between the Compactness Theorem and the space of structures and the space of  $n$-types is fundamental. A logic satisfies the Compactness Theorem if and only if its space of structures is compact, and if and only if its spaces of $n$-types are compact.
To formulate this precisely, we would need to formulate  what a logic is, as is done for example in \cite{Ebbinghaus2016}. We won't do this here; the following two theorems will easily be seen to apply to whatever logics are considered in this paper.

\begin{thm}\label{thm:SATcpct}
Given a logic and a language $L$, every finitely satisfiable set of $L$-formulas
is satisfiable if and only if the space of
$L$-structures $\strl$ is compact.
\end{thm}

\begin{thm}\label{thm:typesat}
Given a logic, a language $L$, every finitely satisfiable set of $L$-formulas
is satisfiable if and only if all the spaces of types $\s^n$ are compact.
\end{thm}

A proof of Theorem \ref{thm:SATcpct} can be found in \cite{Hamel}. A proof of Theorem \ref{thm:typesat} can be found in \cite{TentZiegler}.

\section{$C_p$-theory}
Again, we provide a brief introduction. We shall provide some proofs which are not difficult but are difficult to extract from the literature. 

\subsection{Some Basic Results}
We shall assume basic knowledge of general topology such as can be found in any of the standard texts for a first graduate course, such as \cite{Willard2012}. Unless otherwise stated, we assume that all the topological spaces we consider are infinite, completely regular, and Hausdorff.

\begin{defi}
	Let $X$ be a completely regular Hausdorff topological space. $C_p(X)$ is the space of real-valued continuous functions on $X$, endowed with the subspace topology inherited from the product topology on $\reals^X$. Thus, the basic open sets of $C_p(X)$ are just the basic open sets of $\reals^X$ intersected with $C_p(X)$, i.e.,
	\begin{linenomath*}
	\[W(g;x_0, \ldots, x_n; \epsi) = \st{f \in C_p(X)}{(\forall i \leq n)(\abs{f(x_i) - g(x_i)} < \epsi)}.\]
	\end{linenomath*}
\end{defi}

\begin{rmk*}
	It follows from the form of the basic open sets of $\reals^X$ and complete regularity that $C_p(X)$ is dense in $\reals^X$. We shall occasionally deal with $C_p(X, [0, 1])$ and $C_p(X, 2)$, with the obvious meanings. Notice that each of these is a closed subspace of $C_p(X)$. Since the main properties we deal with are closed-hereditary, if we prove results for $C_p(X)$, they apply also to the other two cases. In particular, this means that our results can be applied to first-order logic, in which $C_p(X, 2)$ is relevant.
\end{rmk*}

\begin{lem}
\label{lem:rstrct}
	Let $X$ be any topological space, $Y \subseteq X$, and $\pi_Y: C_p(X) \to C_p(Y)$ the restriction map given by $\pi_Y(f) = \rstrct{f}{Y}$, $f$ restricted to $Y$. Then:
	\begin{enumerate}[label=(\roman*)]
		\item $\pi_Y$ is continuous.
		\item $\clsr{\pi_Y\left(C_p(X)\right)} = C_p(Y)$.
		\item If $Y \subseteq X$ is dense, then $\pi_Y: C_p(X) \to \pi_Y\left(C_p(X)\right)$ is a continuous bijection.
	\end{enumerate}
\end{lem}

\begin{proof}
See \cite{ArhangelskiiFunction}.
\end{proof}

The following results sample the interrelations between the topology of $X$ and the topology of $C_p(X)$. First, we introduce some topological cardinal functions.

\begin{defi}
	Let $(X, \mcal{T})$ be a topological space.
\begin{enumerate}[label=(\alph*)]
	\item A \emph{network} for $X$ is a family $\mathcal{B}$ of sets such that $(\forall x \in X)(\forall U \in \mcal{T})[(x \in U) \implies (\exists V \in \mathcal{B}) (x \in V \subseteq U)]$. The \emph{network weight} $n(X)$ of $X$ is the minimal cardinality of a network of $X$. A \emph{base} is a network for open sets. The \emph{weight} $w(X)$ of $X$ is the minimal cardinality of a base for $X$.

	\item A \emph{local base of $X$ at a point $x \in X$} is a family $\mathcal{B}_x$ of open sets such that $x \in \Insect \mathcal{B}_x$ and for any open set $U$ containing $x$, there is some $V \in \mathcal{B}_x$ such that $x \in V \subseteq U$. The \emph{character of $X$ at $x$}, denoted $\chi(x,X)$, is the minimal cardinality of a local base of $X$ at $x$. The \emph{character of $X$} is defined by $\chi(X) = \sup\st{\chi(x,X)}{x \in X}$.
\end{enumerate}
\end{defi}

\begin{thm}
\label{thm:crds}
	Given an infinite topological space $X$, the following equalities hold:
	\begin{linenomath*}
	\[\crd{X} = \chi(C_p(X)) = w(C_p(X)).\]
	\end{linenomath*}
\end{thm}

\begin{proof}
	See \cite{ArhangelskiiFunction}.
\end{proof}

We need some other definitions:
\begin{defi}\label{ccpt}
Let $X$ be a topological space.
\begin{enumerate}[label=(\alph*)]
	\item\label{def:relcpct} A subset $A \subseteq X$ is \emph{relatively compact} (in $X$) if $\clsr{A}$ is compact.

	\item\label{def:cntcpct} $X$ is \emph{countably compact} if every infinite subset $Y$ of $X$ has a \emph{limit point}, i.e., a point $x$ such that every open set about $x$ contains infinitely many points of $Y$. If $A \subseteq B \subseteq X$, $A$ is \emph{countably compact in $B$} if every infinite subset of $A$ has a limit point in $B$. $A$ is \emph{relatively countably compact} if its closure in $X$ is countably compact.

	\item A set of real-valued functions $A \subseteq \reals^X$ is \emph{pointwise bounded} if for every $x \in X$ there is an $M_x>0$ such that for all $f \in A$, $\abs{f(x)} < M_x$.
	
	\item The \emph{\lind number $L(X)$ of $X$} is the smallest infinite cardinal such that every open cover includes a subcover of cardinality $\leq L(X)$. $X$ is \emph{\lind} if $L(X) = \aleph_0$.
\end{enumerate}
\end{defi}

Definition \ref{ccpt} (b) can be restated in terms that justify the choice of words for "countable compactness":

\begin{prop}
    A $T_1$ space $X$ is countably compact if and only if every countable open cover of $X$ includes a finite subcover.
\end{prop}

\begin{proof}
	See any general topology text.
\end{proof}

\begin{rmk*}
  We are following \cite{Koenig1987} in Definition \ref{ccpt}\ref{def:relcpct} and \ref{ccpt}\ref{def:cntcpct} above so that we can quote their results. An unfortunate difficulty is that, in contrast to Definition \ref{ccpt}\ref{def:relcpct}, if $A$ is relatively countably compact in $X$, its closure need not be countably compact, although the converse holds. The implication does hold for normal spaces $X$ -- see Lemma \ref{clsrcountablycmpct} below.
\end{rmk*}

Now we state \emph{Grothendieck's Theorem} \cite{Grothendieck1952}, which relates properties of $X$ to properties of $C_p(X)$. It can be considered the first theorem of $C_p$-theory. A proof can be found, for example, in \cite{TodorcevicTopics}.

\begin{thm}
	Let $X$ be a countably compact topological space, and $A \subseteq C_p(X)$. Then $A$ is relatively compact (in $C_p(X)$) if and only if it is countably compact in $C_p(X)$.
\end{thm}

\begin{defi}
	A \emph{$g$-space} is a topological space $X$ satisfying the following property: for every $A \subseteq X$, if $A$ is countably compact in $X$ then $A$ is relatively compact. We say that $X$ is a \emph{hereditary $g$-space} if every subspace of $X$ is a $g$-space.
\end{defi}

The Fr\'echet-Urysohn property arises naturally in the context of $g$-spaces.

\begin{defi}
	Let $X$ be a topological space. $X$ is \emph{Fr\'echet-Urysohn} if whenever $A \subseteq X$, every point $p$ in $\clsr{A}$ is the limit of a sequence of points in $A$, i.e., every open set about $p$ contains all but finitely many members of the sequence.
\end{defi}

The following proposition from \cite{Arhangelskii1997a} is useful for generalizing Grothendieck's Theorem:

\begin{prop}
\label{prop:gspace}
	A $g$-space $Y$ is a hereditary $g$-space if and only if every compact subspace of $Y$ is Fr\'echet-Urysohn.
\end{prop}

Hereditary $g$-spaces are also called \emph{angelic}.

The following definition is the starting point for generalizing \groth's Theorem to a broader class of spaces:

\begin{defi}[\cite{Arhangelskii1997a}]
	A topological space $X$ is \emph{Grothendieck} if $C_p(X)$ is a hereditary $g$-space. $X$ is \emph{weakly Grothendieck} if $C_p(X)$ is a $g$-space.
\end{defi}

Any uncountable discrete space is an example of a space that is weakly Grothendieck but not Grothendieck \cite{Arhangelskii1998}.

The property of being a Grothendieck space is transferable from a dense subspace to the full space. This result appears in \cite{Arhangelskii1997a}.

\begin{thm}
\label{thm:densesubspace}
	Let $Y$ be a dense subspace of $X$. If\, $Y$ is Grothendieck, so is $X$.
\end{thm}

\arhan \cite{Arhangelskii1997a} proved the preservation of Grothendieck spaces under continuous images:

\begin{thm}
\label{thm:contimage}
	The continuous image of a Grothendieck space is Grothendieck.
\end{thm}

The previous two results rely on $C_p(X)$ being a hereditary $g$-space, so the same proof does not work if we merely assume that the space is weakly Grothendieck. Indeed, every space is a continuous image of a discrete---hence weakly \groth---space, but not every space is weakly \groth.

\subsection{Grothendieck spaces}
\begin{defi}\label{def:pSspaces}
	Let $X$ be a topological space.
	\begin{enumerate}[label=(\alph*)]
		\item $X$ is a \emph{\lind $p$-space} if and only if it can be perfectly mapped onto a space with a countable base. Recall that a map is \emph{perfect} if it is continuous, closed, and the preimages of points are compact.

		\item $X$ is a \emph{\lind $\Sigma$-space} if and only if it is the continuous image of a \lind $p$-space.
	\end{enumerate}
\end{defi}

These are not the original definitions but rather, equivalent ones. The definitions of $p$-spaces and $\Sigma$-spaces are quite complicated.

The class of spaces which are \groth has been widely studied by \arhan \cite{Arhangelskii1997a}who proved that
\begin{thm}\label{thm:lindSgroth}
	All \lind $\Sigma$-spaces are \groth.
\end{thm}
\lind $\Sigma$-spaces constitute a nice class as they can be also described as follows: The class of \lind $\Sigma$-spaces is the smallest class of spaces containing all compact spaces, all second countable spaces, and that is closed under finite products, closed subspaces, and continuous images (\cite{Tkachuk2010} contains a detailed exposition of these facts).

Examples of \lind $\Sigma$-spaces include compact spaces, Polish spaces, and $K$\textit{-analytic spaces} (see \cite{Tkachuk2010}). The proof that all \lind $\Sigma$-spaces are \groth is not easy. A reasonably self-contained proof can be found in \cite{Hamelb}. It turns out, however, that for the particular applications to undefinability of Banach spaces that we are mainly interested in, we only need to prove that \emph{countable} spaces are \groth! The proof of this is quite easy so we give it below. The reason we can get away with this trivial case is that we mainly deal with countable languages, such as first-order continuous logic and countable fragments of continuous $\mcal{L}_{\omega_1, \omega}$.

\begin{defi}
  A topological space $X$ is \emph{$\sigma$-compact} if it can be written as a countable union of compact subspaces. $X$ is \emph{k-separable} if it has a dense $\sigma$-compact subspace.
\end{defi}

The following theorem collects some well-known (to $C_p$-theorists) results on conditions that imply Grothendieck:

\begin{thm}\label{whatgrot}
	A topological space $X$ is a \groth space if it satisfies any of the following:
	\begin{enumerate}[label=(\roman*)]
		\item $X$ has a dense countably compact subspace.
		\item $X$ is $k$-separable.
		\item \label{groth3}$X$ has a dense \lind $\Sigma$-space.
	\end{enumerate}
\end{thm}

\begin{proof}\mbox{}
\begin{enumerate}[label=(\roman*)]
	\item \groth's Theorem implies that countably compact spaces are weakly \groth. Pryce \cite{Pryce1971} proved that they are indeed \groth. The result follows from Theorem \ref{thm:densesubspace}.

	\item Follows from \ref{groth3} (see \cite{{Tkachuk2010}}) and Theorem \ref{thm:densesubspace}.

	\item This follows immediately from Theorems \ref{thm:densesubspace} and \ref{thm:lindSgroth}. \qedhere
\end{enumerate}
\end{proof}

The preceding results demonstrate that \groth spaces are wide-ranging and well-behaved. They will probably find other applications in model theory. However, we have realized that for our purposes, the much weaker notion of weak Grothendieckness often suffices. 

\begin{defi}
 The \emph{tightness $t(X)$ of $X$} is the smallest infinite cardinal such that whenever $A \subseteq X$ and $x \in \clsr{A}$, there exists a $B \subseteq A$ such that $\crd{B} \leq t(X)$ and $x \in \clsr{B}$. When $t(X)=\aleph_0$, we say that $X$ is \textit{countably tight}.
\end{defi}

Countable tightness plays an important role in the study of \groth spaces, e.g., in the proof that \lind $\Sigma$-spaces are \groth. There, it is the countable tightness of $C_p(X)$ that is important, and this is obtained from the fact that finite (in fact, countable) products of \lind $\Sigma$-spaces are \lind, see \cite{ArhangelskiiFunction} or \cite{Hamelb}. For us here, the crucial fact is that

\begin{thm}\label{thm:stightwgroth}
	Countably tight spaces are weakly \groth. 	
\end{thm}

Theorem \ref{thm:stightwgroth} is proved in \cite{Arhangelskii1997a} as a part of a more general result. We will prove the specific result here. A different proof appears in \cite{TallGrot}.\\

\begin{proof}
  We need a series of easy lemmas.

\begin{lem}
	In a countably tight space, a set is closed if and only if its intersection with every separable closed set is closed.
\end{lem}

\begin{proof}
	Let $X$ be a countably tight space, and $H \subseteq X$. If $H$ is not closed, there is an $h \in \clsr{H}\setminus H$ and a countable $S \subseteq H$ such that $h \in \clsr{S}$. Then $H \cap \clsr{S}$ is not closed.
\end{proof}

\begin{lem}\label{lem:ctrest}
	A mapping from a countably tight space $X$ to a space $Z$ is continuous if and only if for any countable $T \subseteq X$, the restriction $f|_{\clsr{T}}$ is continuous.
\end{lem}

\begin{proof}
	Take any closed $B \subseteq Y$. We claim that $f^{-1}(B)$ is closed. To see this, suppose $x \in \clsr{f^{-1}(B)}$. Then there is a countable $S \subseteq f^{-1}(B)$ such that $x \in \clsr{S}$. Then $f(x)$ is in the closure of $(f|_{\clsr{S}})(\clsr{S})$ in $\clsr{f(S)}$ by continuity of $f|_{\clsr{S}}$. But $f(S) \subseteq B$, which is closed, so $f(x) \in B$, so $x \in f^{-1}(B)$, which is what we wanted to show.
\end{proof}

\begin{lem}\label{lem:ccimages}
	Suppose $Z$ is countably compact in $X$, and $f: X \to Y$ is continuous. Then $f(Z)$ is countably compact in $Y$.
\end{lem}

\begin{proof}
	Suppose $B$ is an infinite subset of $f(Z)$. Let $Z'$ be obtained by picking one point from each $f^{-1}(\{b\})$, for each $b \in B$. Then $Z'$ is infinite and so has a limit point $p \in X$. Then $f(p)$ is a limit point of $B$, since if $U$ is an open set about $f(p)$, then $f^{-1}(U)$ is an open set about $p$ and so contains some $a \in Z'$. Then $U$ contains $f(a)$, which is an element of $B$. 
\end{proof}

\begin{lem}\label{clsrcountablycmpct}
    If $X$ is a normal space and $A\subseteq X$ is countably compact in $X$, then $\clsr{A}$ is countably compact. 
\end{lem}

\begin{proof}
    Suppose that $\clsr{A}$ is not countably compact. Then there is a countably infinite closed discrete subset $B$ of $\clsr{A}$. Let $B = \{b_n\}_{n<\omega}$. It suffices to find disjoint open sets $\{U_n\}_{n < \omega}$, $b_n \in U_n$, for then we can use normality to obtain disjoint open sets $U, W$, $U\supseteq B$, $V\supseteq X\setminus \Union_{n < \omega}U_n$. Then the open sets $\{U \cap U_n\}_{n < \omega}$ form a discrete collection, i.e., each point in $X$ is at most one of these. Take $a_n \in U\cap U_n \cap A$. Then $\{a_n : n < \omega\}$ has no limit points in $X$, contradicting $A$ being countably compact in $X$. To find the required disjoint open sets, by normality find disjoint open sets $V_n$ containing $b_n$ and $V'_n \supseteq \{b_k : k\neq n\}$. Let $U_n = V_n \cap \Insect\{V'_k : k < n\}$.
\end{proof}

\textit{Proof of Theorem \ref{thm:stightwgroth}.}
	Let $X$ be a countably tight space, and let $A$ be countably compact in $C_p(X)$. Since $C_p(X)$ is canonically embedded in $\reals^X$, the set $A$ is also a subspace of $\reals^X$. By Lemma \ref{lem:ccimages}, the images of $A$ under the natural projections of $\reals^X$ onto $\reals$ are also countably compact in $\reals$. Since $\reals$ is metrizable, it is normal, so the images of $A$ have countably compact closures, but then by metrizability they have compact closures. Since $A$ is included in the topological product of these projections, the closure of $A$ in $\reals^X$ is a compact subspace $F$ of $\reals^X$.

	It remains to show $F \subseteq C_p(X)$, i.e., each member $f$ of $F$ is continuous. By Lemma \ref{lem:ctrest}, it suffices to show that for every countable $Y \subseteq X$ the restriction $f|_{\clsr{Y}}$ is continuous. The restriction map $r: \reals^X \to \reals^{\clsr{Y}}$ sends $A$ to a dense subset of $r(F)$, since $r$ is continuous and $A$ is dense in $F$. Since $r$ is continuous, $r(A)$ is countably compact in $C_p(\clsr{Y})$. As a separable space, $Y$ is \groth by Theorem \ref{whatgrot} (or by Theorem \ref{thm:grothcntbl} plus Theorem \ref{thm:densesubspace}) and so the closure of $r(A)$ in $C_p(\clsr{Y})$ is compact. But then that closure must equal $r(F)$, since, by the continuity of $r$, $r(F)$ is the closure of $r(A)$ in $\reals^{\clsr{Y}}$, so is the smallest closed set including $r(A)$ there. Thus, for each $f \in F$ and each countable $Y \subseteq X$, the restriction $f|_{\clsr{Y}}$ is an element of $C_p(\clsr{Y})$, i.e., $f|_{\clsr{Y}}$ is continuous. Since $X$ is a countably tight space, it follows that $f$ is continuous.
\end{proof}

As we said, all we will need in order to prove our definability result is

\begin{thm}\label{thm:grothcntbl}
	Every countable space is \groth.
\end{thm}

\begin{proof}
	By Theorem \ref{thm:crds}, if $X$ is countable, $C_p(X)$ has a countable base. Having a countable base is hereditary, so we need only show that a space with a countable base is a $g$-space. A space with a countable base is separable metrizable, so normal, so by Lemma \ref{clsrcountablycmpct}, if $A$ is countably compact in $C_p(X)$, then $\clsr{A}$ is countably compact. But $C_p(X)$ is \lind, so $\clsr{A}$ is compact.
\end{proof}

\subsection{Grothendieck Spaces and Double Limit Conditions}
The relationship of \groth's Theorem to double limit conditions was explored by V.~Pt\'ak \cite{Ptak1963}:

\begin{thm}
	Let $X$ be compact and $A \subseteq C_p(X)$ pointwise bounded. Then $A$ is relatively compact in $C_p(X)$ if and only if the following double limit condition holds: whenever $\seq{x}{n}{n < \omega}$ is a sequence in $X$ and $\seq{f}{m}{m<\omega}$ is a sequence in $A$, the double limits $\lim_{n \to \infty}\lim_{m \to \infty}f_m(x_n)$ and $\lim_{m \to \infty}\lim_{n \to \infty}f_m(x_n)$ are equal whenever they exist.
\end{thm}

Note that if $A$ is countably compact in $C_p(X)$, then it is pointwise bounded. The relationship among double limit conditions and countable compactness and compactness of function spaces was further explored by several authors, most notably H.~K\"onig and N.~Kuhn \cite{Koenig1987}. A topological variant of such sequential double limits was motivated by model-theoretic considerations and will appear frequently in what follows.

\begin{defi}
	Let $X$ be a topological space. Given an ultrafilter $\mcal{U}$ on a regular cardinal $\kappa$, and a $\kappa$-sequence $\seq{x}{\alpha}{\alpha < \kappa}$ in $X$, we say that
	\begin{linenomath*} 
	\[\lim_{\alpha \to \mcal{U}}x_\alpha = x\]
	\end{linenomath*} 
	if and only if for every open neighbourhood $U$ of $x$, $\st{\alpha < \kappa}{x_\alpha \in U} \in \mcal{U}$.
\end{defi}

It follows from the fact that an ultrafilter cannot contain disjoint sets that if such an ultralimit exists in a Hausdorff space, it is unique. The following result relates ultralimits and compactness:

\begin{thmr}[\cite{Iovino1999,Hamelb}]\label{ultracomp}
	A space $X$ is compact if and only if every ultralimit in $X$ exists.
\end{thmr}

In the following sections we shall present generalizations of the main results of \cite{Casazza} by working with more general spaces rather than just compact spaces. For this purpose we introduce some definitions which are different from the ones in \cite{Casazza} but identical with the added assumption of compactness.

\section{Stability, Definability, and Double (Ultra)limit Conditions}
Thorough introductions to continuous logic can be found in \cite{Eagle2014}, \cite{Eagle2015} and \cite{BenYaacov2008}. In \cite{Hamel}, the authors introduced in detail a setting for continuous logic, mostly following \cite{Casazza} and \cite{Eagle2015}. It is worth noting that there are different definitions of continuous logics in the literature. However, most authors differ only slightly from one another. We shall now introduce \emph{metric structures} following Eagle \cite{Eagle2014} and \cite{Eagle2015} as this is the context we shall work in:

\begin{defi}\mbox{}
\begin{enumerate}[label=(\alph*)]
	\item If $(M,d)$ and $(N,\rho)$ are metric spaces and $f\colon M^n\to N$ is uniformly continuous, a \textit{modulus of uniform continuity of} $f$ is a function 
    $\delta \colon (0,1)\cap{\mathbb{Q}} \to (0,1)\cap{\mathbb{Q}}$ such that whenever $\boldsymbol{a}=(a_1,...,a_n),\ \boldsymbol{b}=(b_1,...,b_n)\in M^n$ and $\varepsilon \in (0,1)\ \cap\ {\mathbb{Q}}$, $\sup\{d(a_i,b_i) : 1\leq i \leq n\}<\delta(\varepsilon)$ implies
    $\rho(f(\boldsymbol{a}),f(\boldsymbol{b}))<\varepsilon$. Similarly define a modulus of uniform continuity for a predicate.

	\item A \emph{language} for metric structures is a set $L$ which consists of constants, functions with an associated arity, and a modulus of uniform continuity; predicates with an associated arity and a modulus of uniform continuity; and a symbol $d$ for a metric.

	\item An \emph{$L$-metric structure} $\mfrak{M}$ is a metric space $(M, d^M)$ together with interpretations for each symbol in $L$: $c^{\mfrak{M}} \in M$ for each constant $c \in L$; $f^{\mfrak{M}}: M^n \to M$ a uniformly continuous function for each $n$-ary function symbol $f \in L$; $P^{\mfrak{M}}: M^n \to [0, 1]$ a uniformly continuous function for each $n$-ary predicate symbol $P \in L$. Assume for the sake of notational simplicity that all metric structures have diameter $1$. 

	\item The space of $L$-structures in a given language $L$ is the family $\mathrm{Str}(L)$ of all $L$-metric structures endowed with the topology generated by the following basic closed sets: $[\vphi] = \st{\mfrak{M}}{\mfrak{M} \ent \vphi}$ where $\vphi$ is an $L$-sentence.
\end{enumerate}
\end{defi}

\begin{rmk*}
  We should of course consider equivalence classes as we did in Definition \ref{def:eqclass}, but we instead follow the practice of model theorists who work in this arena. The reader should rest assured that the sloppiness can be fixed.
\end{rmk*}

\begin{rmk*}
	The continuous first-order formulas are constructed just as in the discrete setting except for the following addition: if $f: [0, 1]^n \to [0, 1]$ is a continuous function and $\vphi_0, \ldots, \vphi_{n-1}$ are $L$-formulas, then $f(\vphi_0, \ldots, \vphi_{n-1})$ is also an $L$-formula. It is customary to identify formulas $\vphi$ of arity $n$ with functions $\mf{M}^n \to [0, 1]$. This allows an easy way to define the satisfaction relation, i.e., if $\vphi$ is an $L$-formula, $\mf{M}$ an $L$-structure, and $a \in \mf{M}^n$, then $\mf{M} \ent \vphi(a)$ if and only if $\vphi(a) = 1$. For example, under this identification, the conjunction, $\vphi \lnd \psi$, of two $L$-formulas $\vphi$ and $\psi$ is $\max(\vphi, \psi)$. In what follows, we shall introduce other continuous logics and study Gowers' problem in them; our main example is continuous $\mcal{L}_{\omega_1, \omega}$ as introduced by Eagle \cite{Eagle2014}, \cite{Eagle2015}. See \cite{Eagle2014}, \cite{Eagle2015} and \cite{Hamel} for a detailed exposition.
\end{rmk*}

Now we will tie model theory and $C_p$-theory together and greatly improve the results of \cite{Casazza}, \cite{Hamel}, and \cite{Hamelb}. A key technical concept introduced in \cite{Casazza} and used to talk about a pair of Banach space norms is the following:

\begin{defi}\mbox{}
\begin{enumerate}[label=(\alph*)]
	\item Let $L$ be a language $L' \supseteq L$ is a \emph{language for pairs of structures from $L$}, if $L'$ includes two disjoint copies of $L$ and there is a map $\mr{Str}(L) \times \mr{Str}(L) \to \mr{Str}(L')$ which assigns to every pair of $L$-structures $\mf{M}, \mf{N}$ an $L'$-structure $\langle\mf{M}, \mf{N} \rangle$.

	\item Let $L'$ be a language for pairs of structures from $L$, and $X, Y$ function symbols from $L$. We say that a formula $\vphi(X, Y)$ is a \emph{formula for pairs of structures from $L$} if
	\begin{linenomath*}  
	\[(\mf{M}, \mf{N}) \mapsto \vphi(X^{\mf{M}}, Y^{\mf{N}}) = \mr{Val}\left(\vphi(X, Y), \langle\mf{M}, \mf{N} \rangle\right)\]
	\end{linenomath*} 
	is separately continuous on $\mr{Str}(L) \times \mr{Str}(L)$. For simplicity, we write $\vphi(\mf{M}, \mf{N})$ instead of $\mr{Val}\left(\vphi(X, Y), \langle\mf{M}, \mf{N} \rangle\right)$.
\end{enumerate}
\end{defi}

In this section we shall only be interested in $\vphi$-types when $\vphi$ is a formula for pairs of structures:

\begin{defi}
	Let $L$ be a language, $\vphi$ an $L$-formula for pairs of structures, and $\mf{M} \in \mr{Str}(L)$. The \emph{left $\vphi$-type of $\mf{M}$} is the function $\mr{ltp}_{\vphi, \mf{M}}: \mr{Str}(L) \to [0, 1]$ given by $\mr{ltp}_{\vphi, \mf{M}}(\mf{N}) = \vphi(\mf{M}, \mf{N})$.

	The \emph{space of left $\vphi$-types}, denoted $S^{l}_{\vphi}$ is the closure of $\st{\mr{ltp}_{\vphi, \mf{M}}}{\mf{M} \in \mr{Str}(L)}$ in $C_p(\mr{Str}(L), [0, 1])$. If $C$ is a subset of $\mr{Str}(L)$, then $S^{l}_{\vphi}(C)$ is the closure of $\st{\rstrct{\mr{ltp}_{\vphi, \mf{M}}}{C}}{\mf{M} \in C}$ in $C_p(C, [0, 1])$ and it is called the \emph{space of left $\vphi$-types over $C$}.

	The definitions for \emph{right $\vphi$-types} are analogous.
\end{defi}

Notice that the previous definition coincides with the classical one, that appears, for instance, as Definition 2.5 in \cite{Iovino1999}, and with the one given in \cite{Casazza} for the compact case: if the logic is compact, that is to say $\mr{Str}(L)$ is compact, then, using \ref{ultracomp} on $\mr{Str}(L)$ and the separate continuity of $\varphi$, $S^l_{\vphi}$ turns out to be the closure of $\{\mr{ltp}_{\vphi, \mf{M}} : \mf{M} \in \mr{Str}(L)\}$ in $[0, 1]^{\mr{Str}(L)}$, which is then also compact. Since compact spaces are \groth, we have that $S^l_{\vphi}$ is always Fr\'echet-Urysohn when the logic is compact---see Proposition \ref{prop:gspace}. 

The alternative approach to talk about (left) $\varphi$-types in the compact case which is widely available in the literature is by considering maximal satisfiable sets of formulas of the form $\varphi(x,-)$. See for instance Definition \ref{classictype} or the complete presentation in Unit 4.2 of \cite{TentZiegler}. This approach makes it harder for us to exploit the topological richness of these spaces of continuous functions. However, one can always recover the $\varphi$-type in this other classical sense from our approach: take the inverse image $(\mr{ltp}_{\vphi, \mf{M}})^{-1}\{1\}$, as this will yield the set of all $\mf{N}$ for which $\vphi(\mf{M}, \mf{N})$ holds.

For the sake of completeness, we include a short proof of the compactness of the space of left $\varphi$-types for a compact logic. 

\begin{prop}
If a logic is compact (i.e. $\strl$ is compact), then $S^{l}_{\vphi}$ is compact.
\end{prop}
\begin{proof}
If $f$ is a limit point of $\{\mr{ltp}_{\vphi, \mf{M}} : \mf{M} \in \mr{Str}(L)\}$, then there is a cardinal $\kappa$, an ultrafilter $\mcal{U}$ over $\kappa$ and a sequence $\langle{\mf{M}_{\alpha} \colon\alpha<\kappa}\rangle$ in $\mr{Str}(L)$ such that $\lim_{\alpha\to\mcal{U}} \mr{ltp}_{\vphi, \mf{M_\alpha}}=f$. Since $\mr{Str}(L)$ is compact and $\varphi$ is separately continuous, there is an $\mf{M}$ such that $\lim_{\alpha\to\mcal{U}}\mf{M_\alpha}=\mf{M}$ and for any $\mf{N}\in{\mr{Str}(L)}$,
$\lim_{\alpha\to{\mcal{U}}}
\varphi(\mf{M}_{\alpha},\mf{N})=
\varphi(\mf{M},\mf{N})=\mr{ltp}_{\vphi, \mf{M}}\in S^{l}_{\vphi}$.
Then, $S^{l}_{\vphi}$ is a closed subset of 
$[0, 1]^{\mr{Str}(L)}$, and therefore compact.
\end{proof}

In fact, the definition of the space of $\varphi$-types in first-order logic appears sometimes in the literature, see \cite{Iovino1999} for instance, with the closure taken directly in $[0, 1]^{\mr{Str}(L)}$ since, as the end of the previous proof shows, this makes no difference in the first-order case.

For the purposes of this paper, we care about $\varphi$-types of the form $\mr{ltp}_{\vphi, \mf{M}}$, and the approximations we can do by taking ultralimits of $\kappa$-sequences of them. The proof of the previous proposition shows that all elements in $S^{l}_{\vphi}$ are of one of these two forms in the compact case. Similarly, these are the only two possibilities in the countably tight case: the proof of Theorem \ref{thm:doublelimit} below shows that every $\varphi$-type is of the form $\mr{ltp}_{\vphi, \mf{M}}$, or the ultralimit of a (countable) sequence of such elements. Even if a logic were sophisticated enough for $S^{l}_{\vphi}$ to contain elements of a different form, our results will still apply as they are only concerned with the two aforementioned different forms of $\varphi$-types. 

The notion of model-theoretic stability has been extensively studied in discrete model theory, whereas works like \cite{Casazza} or \cite{Hamel} are devoted to understanding stability in the continuous setting. There are many definitions of stability which are equivalent in case the logic under consideration is compact. We have found the one introduced in \cite{Iovino1999} involving a \emph{double ultralimit condition} to be the most suitable for the continuous setting, even when it comes to non-compact logics:

\begin{defi}
  Let $L$ be a language for pairs of structures, $\vphi$ a formula for pairs of structures, and $C \subseteq \mr{Str}(L)$. We say that $\vphi$ is \emph{stable on $C$} if and only if whenever $\seq{\mf{M}}{i}{i<\omega}$ and $\seq{\mf{N}}{j}{j<\omega}$ are sequences in $C$, and $\mcal{U}$ and $\mcal{V}$ are ultrafilters on $\omega$, we have
  \begin{linenomath*} 
  \[\lim_{i \to \mcal{U}}\lim_{j \to \mcal{V}}\vphi(\mf{M}_i, \mf{N}_j) = \lim_{j \to \mcal{V}}\lim_{i \to \mcal{U}}\vphi(\mf{M}_i, \mf{N}_j).\]
  \end{linenomath*} 
\end{defi}

We won't specifically use stability in this chapter, but this version of stability from first-order logic certainly motivated us. The reader familiar with the model-theoretic notion of stability will realize that the definition above is not the most common one. Usually, stability is presented combinatorially through the order property or in terms of the cardinalities of the space of types, see \cite{TentZiegler}. Any introductory textbook in Model Theory will include a long list of statements that are equivalent to stability in first-order logic; definability, for instance, is in such a list. These equivalences cease to hold once we move beyond compactness. In fact, the reason for our definition of choice for stability is that we are going to present a detailed analysis of the conditions that enable the equivalence between this definition -- via the double (ultra)limit condition -- and that of definability in logics that are not compact. We now state the topological backbone of one of our main results:

\begin{defi}
  Let $X$ be a topological space and $A$ a subset of $C_p(X, [0, 1])$. We write $\udlc(A, X)$, if 
\begin{quote}
  \label{DLC} for every pair of sequences $\seq{f}{n}{n < \omega} \subseteq A$ and $\seq{x}{m}{m < \omega} \subseteq X$, and ultrafilters $\mcal{U}$ and $\mcal{V}$ on $\omega$, the double limits
  \begin{linenomath*} 
  \[\lim_{n \to \mcal{U}}\lim_{m \to \mcal{V}}f_n(x_m) = \lim_{m \to \mcal{V}}\lim_{n \to \mcal{U}}f_n(x_m)\]
  \end{linenomath*} 
  agree whenever $\lim_{m \to \mcal{V}}x_m$ exists.
\end{quote} 
We write $\udlc(X)$ if $\udlc(A, X)$ holds for all $A \subseteq C_p(X, [0, 1])$. We say $X$ satisfies the \emph{double ultralimit condition} if for each $A \subseteq C_p(X, [0, 1])$, $\udlc(A, X)$ is equivalent to $A$ being relatively countably compact.
\end{defi}

The inspiration for the definition of the double ultralimit condition comes from the double limit condition used by Pták that we recast as follows:

\begin{defi}
  Let $X$ be a topological space and $A$ a subset of $C_p(X, [0, 1])$. We write $DLC(A, X)$, if 
\begin{quote}
  \label{DLCreg} for every pair of sequences $\seq{f}{n}{n < \omega} \subseteq A$ and $\seq{x}{m}{m < \omega} \subseteq X$, the double limits
  \begin{linenomath*} 
  \[\lim_{n \to \infty}\lim_{m \to \infty}f_n(x_m) = \lim_{m \to \infty}\lim_{n \to \infty}f_n(x_m)\]
  \end{linenomath*} 
  agree whenever the limits exist.
\end{quote} 
We write $DLC(X)$ if $DLC(A, X)$ holds for all $A \subseteq C_p(X, [0, 1])$. We say $X$ satisfies the \emph{double limit condition} if for each $A \subseteq C_p(X, [0, 1])$, $DLC(A, X)$ is equivalent to $A$ being relatively countably compact.
\end{defi}

As an immediate consequence of these definitions, we have

\begin{lem}
  If $X$ is weakly \groth and satisfies the double ultralimit condition, then a subset $A \subseteq C_p(X, [0, 1])$ satisfying $\udlc(A, X)$ is relatively compact.
\end{lem}

For countably tight spaces, the ``if'' can be strengthened to ``if and only if'':

\begin{thm}
\label{thm:doublelimit}
	Let $X$ be countably tight. A subset $A$ of $C_p(X, [0, 1])$ is relatively compact in $C_p(X, [0, 1])$ if and only if $\udlc(A, X)$. 
\end{thm}

\begin{proof}
	Let $\clsr{A}$ denote the closure of $A$ in $[0, 1]^X$. Suppose $\clsr{A} \cap C_p(X)$ is not compact in $C_p(X)$. Then $\clsr{A}\, \cap\, C_p(X)$ is closed but it is not countably compact in $C_p(X)$, since $X$ is weakly \groth. Let $\seq{f}{n}{n < \omega}$ be a subset of $A$ with closure disjoint from $\clsr{A} \cap C_p(X)$. Since $\clsr{A}$ is a compact subset of $[0, 1]^X$, each ultralimit of the sequence $\seq{f}{n}{n < \omega}$ exists, and is discontinuous. Take a non-principal ultrafilter $\mcal{U}$ over $\omega$ and let $\lim_{n \to \mcal{U}}f_n = g$, where $g$ is discontinuous by assumption. Then there are $\epsi>0$ and $y \in X$ such that $y \in \clsr{Y}$, where
	\begin{linenomath*}  
	\[Y = X \setminus g^{-1}\left(g(y)-\epsi, g(y)+\epsi\right).\]
	\end{linenomath*} 
	Since $t(X) = \aleph_0$, there is some $Z \subseteq Y$ with $\crd{Z} = \aleph_0$ and $y \in \clsr{Z}$. Suppose $Z = \seq{x}{m}{m < \omega}$ and for each open neighbourhood $U$ of $y$, let $M_U = \st{m < \omega}{x_m \in U}$. Clearly, the family of all $M_U$ is centred (i.e. all finite subfamilies have non-empty intersections) and so it can be extended to an ultrafilter $\mcal{V}$ on $\omega$ so that
	\begin{linenomath*} 
	\[\lim_{n \to \mcal{U}}\lim_{m \to \mcal{V}}f_n(x_m) = g(y)\]
	\end{linenomath*} 
	since each $f_n$ is continuous. On the other hand,
	\begin{linenomath*} 
	\[\lim_{m \to \mcal{V}}\lim_{n \to \mcal{U}}f_n(x_m) = \lim_{m \to \mcal{V}}g(x_m)\]
	\end{linenomath*} 
	exists by compactness of $[0, 1]$. However, by the choice of each $x_m$, we have $\abs{g(y) - g(x_m)} > \epsi$ and so the ultralimits exist but are different, a contradiction.

	Conversely, suppose $\clsr{A}\, \cap\, C_p(X)$ is compact and $\lim_{m \to \mcal{V}}x_m = y$. Then for any sequence $\seq{f}{n}{n < \omega} \subseteq A$ and ultrafilter $\mcal{U}$ on $\omega$, there is a continuous function $g = \lim_{n \to \mcal{U}}f_n$. Thus,
	\begin{linenomath*} 
	\[\lim_{n \to \mcal{U}}\lim_{m \to \mcal{V}}f_n(x_m) = \lim_{n \to \mcal{U}}f_n(y) = g(y),\]
	\end{linenomath*} 
	and
	\begin{linenomath*} 
	\[\lim_{m \to \mcal{V}}\lim_{n \to \mcal{U}}f_n(x_m) = \lim_{m \to \mcal{V}}g(x_m) = g(y). \qedhere\]
	\end{linenomath*}
 Note we did not need countable tightness for this direction.
\end{proof}

The main application we consider here deals with countable fragments of $\mcal{L}_{\omega_1, \omega}$, for the corresponding space of types is Polish \cite{Morley1974}, so (\groth and) countably tight.

The reader will naturally wonder about the relationship between (sequential) double limit conditions and double ultralimit conditions. We will now provide an answer. We begin with an observation: if $\lim_{n\rightarrow\infty}a_n=L$, then $\lim_{n\rightarrow\mathcal{U}}a_n=L$ for any non-principal $\mathcal{U}$ on $\omega$. The converse is, of course, false. 

Let's state four conditions for a topological space $X$ and $A \subseteq C_p(X,[0,1])$. 

\begin{enumerate}[label=(\textbf{\Roman*})]
\item\label{relcpct} $A$ is relatively compact in $C_p(X, [0, 1])$.

\item\label{relcntbl} $A$ is relatively countably compact in $C_p(X, [0, 1])$.

\item\label{DLC} DLC($A, X$).

\item\label{UDLC} $\udlc(A, X)$.
\end{enumerate}

Pt\'ak proved \ref{relcpct} if and only if \ref{DLC} for $X$ compact, i.e. compact spaces satisfy the double limit condition. We proved \ref{relcpct} if and only if \ref{UDLC} for $X$ countably tight, i.e. countably tight spaces satisfy the double ultralimit condition. Of course we have \ref{relcpct} if and only if \ref{relcntbl} if $X$ is weakly \groth. 

\begin{thm}
    If $X$ satisfies the double limit condition, then $X$ satisfies the double ultralimit condition.
\end{thm}

\begin{proof}

That $A$ being relatively compact in $C_p(X, [0, 1])$ implies that the order at which the ultralimits are taken can be changed follows from a proof identical to the one for countable tightness. We now want to show that being able to change the order at which the limits are taken implies that $A$ is relatively compact in $C_p(X, [0, 1])$.

Take sequences $\seq{f}{n}{n < \omega} \subseteq A$ and $\seq{x}{m}{m < \omega} \subseteq X$. We shall use DLC($A, X$) to prove that $A$ is relatively compact. Thus, also assume that the limits $\lim_{n \to \infty}\lim_{m \to \infty}f_m(x_n)$ and $\lim_{m \to \infty}\lim_{n \to \infty}f_m(x_n)$ exist. We want to show that they are equal. 

Since the limits exist, they are equal to any ultralimit along a non-principal ultrafilter over $\omega$. Take such ultrafilters $\mcal{U}, \mathcal{V}$ on $\omega$. Then:

$\lim_{n \to \infty}\lim_{m \to \infty}f_m(x_n)=\lim_{n \to \mathcal{U}}\lim_{m \to \mathcal{V}}f_m(x_n)$ and $\lim_{m \to \mathcal{V}}\lim_{n \to \mathcal{U}}f_m(x_n)\break=\lim_{m \to \infty}\lim_{n \to \infty}f_m(x_n),$

by our observation above. But then the ultralimits can be exchanged because, by DLC($X$), the limits can be.
\end{proof}

Analogously to the discrete compact case, the model-theoretic version of Theorem \ref{thm:doublelimit} will relate definability to a double limit condition. For this purpose, we introduce the notion of definability for $\vphi$-types, where $\vphi$ is a formula for pairs of structures, following \cite{Casazza}:

\begin{defi}
	Suppose $L$ is a language for pairs of structures and $\vphi$ is a formula for pairs of structures. Let $C \subseteq \mr{Str}(L)$. A function $\tau: S^{r}_{\vphi}(C) \to [0, 1]$ is a \emph{left global $\vphi$-type over $C$} if there is a sequence $\seq{\mf{M}}{\alpha}{\alpha < \kappa} \subseteq C$, and an ultrafilter $\mcal{U}$ on $\kappa$, such that for every type $t \in S^{r}_{\vphi}(C)$, say $t = \lim_{\beta \to \mcal{V}}\mr{rtp}_{\vphi,\mf{N}_\beta}$, we have
	\begin{linenomath*} 
	\[\tau(t) = \lim_{\alpha \to \mcal{U}}\lim_{\beta \to \mcal{V}}\vphi(\mf{M}_\alpha, \mf{N}_\beta).\]
	\end{linenomath*} 
	We say that $\tau$ is \emph{explicitly definable} if it is continuous. If a left $\vphi$-type is given by $p = \lim_{\alpha \to \mcal{U}}\mr{ltp}_{\vphi,\mf{M}_\alpha}$, we say that $p$ is \emph{explicitly definable} if the respective $\tau$ is continuous.
\end{defi}

The classic idea of definability is rather simpler than the definition above, it just requires the object to be \textit{defined} by a formula. The previous definition is sound since in continuous logics the allowable formulas are continuous functions. Following the notation of the definition above, that a type $p$ is explicitly definable always implies that the corresponding global type $\tau$ is continuous; this is true in any logic even beyond first-order. On the other hand, that the continuity of $\tau$ implies that $p$ is explicitly definable requires, in general, compactness, usually through an argument related to the Stone-Weierstrass Theorem. We shall make use of two variations on this idea:
\begin{enumerate}
    \item We will show that certain types --- those of the pathological Banach spaces we'll consider --- are not explicitly definable by showing that their corresponding global types are not continuous.
    \item For a positive definability result, we will assume that a certain type is explicitly definable, and use the continuity of the corresponding global type that follows.
\end{enumerate}

Moreover, as with most concepts in continuous logic, such a definition reduces to the usual notion in the discrete setting. Also notice that, although our main application deals with Banach spaces, the context of the previous definition and the results in this section is that of the languages for pairs of structures rather than just Banach spaces. 

We proceed with the model-theoretic version of Theorem \ref{thm:doublelimit}:

\begin{thm}\label{maindlc}
	Let $L$ be a language for pairs of structures, and $\vphi$ a formula for pairs of structures. Suppose $C \subseteq \mr{Str}(L)$ is such that $S^{l}_{\vphi}(C)$ is countably tight. Then the following are equivalent:
	\begin{enumerate}[label=(\roman*)]
		\item \label{equi1} Whenever a left type over $C$ is given by $t = \lim_{i \to \mcal{U}}\mr{ltp}_{\vphi, \mf{M}_i}$, and $\seq{\mf{N}}{j}{j<\omega} \subseteq C$ is a sequence in $C$, and $\mcal{V}$ is an ultrafilter on $\omega$, then
		\begin{linenomath*} 
		\[\lim_{i \to \mcal{U}}\lim_{j \to \mcal{V}}\vphi(\mf{M}_i, \mf{N}_j) = \lim_{j \to \mcal{V}}\lim_{i \to \mcal{U}}\vphi(\mf{M}_i, \mf{N}_j).\]
		\end{linenomath*} 

		\item \label{equi2} Whenever a left type over $C$ is given by $t = \lim_{i \to \mcal{U}}\mr{ltp}_{\vphi, \mf{M}_i}$, and $\seq{\mf{N}}{j}{j<\omega} \subseteq C$ is a sequence in $C$, then
		\begin{linenomath*} 
		\[\sup_{i < j}\vphi(\mf{M}_i, \mf{N}_j) = \inf_{j < i}\vphi(\mf{M}_i, \mf{N}_j).\]
		\end{linenomath*} 

		\item \label{equi3} If $\tau$ is a global left $\vphi$-type over $C$, then $\tau$ is continuous.
	\end{enumerate}
\end{thm}

\begin{proof}
	For a proof of the equivalence of \ref{equi1} and \ref{equi2}, see \cite{Iovino1999} or \cite{Hamel}. To see that \ref{equi1} implies \ref{equi3}, suppose $\langle \mf{N}_{\alpha} : \alpha<\kappa\rangle$ is a $\kappa$-sequence in $C$ such that 
	\begin{linenomath*} 
	\[\tau(\lim_{i \to \mcal{U}}\mr{ltp}_{\vphi, \mf{M}_i}) = \lim_{i \to \mcal{U}}\lim_{\alpha \to \mcal{V}}\vphi(\mf{M}_i, \mf{N}_{\alpha})\]
	\end{linenomath*} 
	defines a global $\vphi$-type, where $\mcal{V}$ is an ultrafilter on $\kappa$. Define $f_{\alpha}: S^l_{\vphi}(C) \to [0, 1]$ by
	\begin{linenomath*} 
	\[f_{\alpha}(\lim_{i \to \mcal{U}}\mr{ltp}_{\vphi, \mf{M}_i}) = \lim_{i \to \mcal{U}}\vphi(\mf{M}_i, \mf{N}_{\alpha}).\]
	\end{linenomath*} 
	Then the set $A = \{f_{\alpha} : \alpha<\kappa\}$ is a subset $A \subseteq C_p(S^l_{\vphi}(C), [0, 1])$ and satisfies $\udlc(A, X)$, so that $\clsr{A} \cap C_p(S^l_{\vphi}(C), [0, 1])$ is compact and $\tau = \lim_{\alpha \to \mcal{V}}f_{\alpha}$ is continuous. That \ref{equi3} implies \ref{equi1} is immediate.
\end{proof}

\begin{rmk}\label{rmk:cpctproved}
  All this proof used of countable tightness was that it implies weakly \groth and the double ultralimit condition. Thus we have also proved the compact case, which was proved in \cite{Casazza}. In Section 6 we shall exploit this observation in order to obtain the countably compact case, which \cite{Casazza} claimed but is no longer claimed in \cite{CDI}.
\end{rmk}

\section{Applications and Examples Concerning the Undefinability of Pathological Banach Spaces}
We first introduce some concepts from Analysis which we will deal with in what follows. We denote the Banach space of sequences of real numbers that are eventually $0$ by $c_{00}$. The space of sequences of real numbers converging to $0$ is denoted by $c_0$. Finally, $\ell^p$ denotes the space of sequences $\seq{x}{n}{n<\omega}$ of real numbers such that $\sum_{n < \omega}\abs{x_n}^p < \infty$.

B.~Tsirelson \cite{Tsirelson1974} constructed a Banach space which does not include a copy of any $\ell^p$ or $c_0$. What is now called \textit{Tsirelson's space} is due to T.~Figiel and W.~Johnson \cite{Figiel1974}: let $\seq{t}{n}{n<\omega}$ be the canonical basis for $c_{00}$. If $x = \sum_{n<\omega}a_nx_n \in c_{00}$ and $E, F \in [\omega]^{<\omega}$, we denote $\sum_{n \in E}a_nx_n$ by $Ex$, and similarly for $F$. We write $E \leq F$ if and only if $\max E \leq \min F$.

Recursively define
\begin{linenomath*} 
\begin{align*}
 	&\norm{x}_0 = \norm{x}_{c_{00}}, \\
 	&\norm{x}_{n+1} = \max\left\{\norm{x}_n, \frac{1}{2}\max_k\left\{\sum_{i<k}\norm{E_ix}_n\,:\,\{k\} \leq E_1 < \cdots < E_k\right\}\right\}.
 \end{align*}
 \end{linenomath*} 
 Then,
 \begin{linenomath*} 
 \[\norm{x}_n \leq \norm{x}_{n+1} \leq \norm{x}_{\ell^1}\]
 \end{linenomath*} 
 and we denote $\norm{x}_T = \lim_{n \to \infty}\norm{x}_n$. Then,
 \begin{linenomath*} 
 \[\norm{x}_T = \max\left\{\norm{x}_{c_{00}}, \frac{1}{2}\max\left\{\sum_{i<k}\norm{E_ix}_T\,:\,\{k\} \leq E_1 < \cdots < E_k\right\}\right\}\]
 \end{linenomath*} 
 and $T$ is the norm completion of $\norm{}_T$. The reader is referred to P.~Casazza and T.~Shura \cite{Casazza1989} for a detailed exposition of Tsirelson's space. 

 The context of the subsequent results is what is now known as \textit{Gowers' problem: does every infinite dimensional explicitly definable Banach space include an isomorphic copy of $\ell^p$ for some $p$, or $c_0$?} We found in \cite{Casazza} the inspiration for the framework and results presented in the current paper. The topological assumption used in \cite{Casazza} is compactness; we have shown countable tightness is just as good. Neither compactness nor countable tightness imply the other. The interval $(0, 1)$ is countably tight but not compact; $\omega_1+1$ is compact but not countably tight. In terms of results, we can trivially deduce their main definability results from ours because finitary continuous logic is a countable fragment (see below) of infinitary continuous logic. However their more general definability results for compact logics were obtained using a double limit condition on sequences. To obtain our results, we needed to instead use a double limit condition involving ultralimits and assume countable tightness on the relevant spaces of $\varphi$-types. However, a common generalization of compactness and countable tightness that yields the desired undefinability results is precisely our double ultralimit condition plus weakly \groth.

At the end of this paper we shall explore topological weakenings of compactness that still yield the desired undefinability results by exploiting \cite{Koenig1987}, but for now, fix a language $L$ for a pair of structures. Let $C$ be the set of all $L$-structures which are normed spaces based on $c_{00}$ (see definition below). We introduce the continuous logic formula for pairs of structures used in \cite{Casazza}: for norms $\norm{}_1$ and $\norm{}_2$, define
 \begin{linenomath*} 
 \[D(\norm{}_1, \norm{}_2) = \sup\left\{\frac{\norm{x}_1}{\norm{x}_2}\,:\,\norm{x}_{\ell^1} = 1\right\}.\]
 \end{linenomath*} 
 Then define
 \begin{linenomath*} 
 \begin{equation}
 \tag{$\star$}
 \label{def:vphi}
 \vphi(\norm{}_1, \norm{}_2) = 1 - \frac{\log D(\norm{}_1, \norm{}_2)}{1+\log D(\norm{}_1, \norm{}_2)}.
 \end{equation}
 \end{linenomath*} 

 It is not difficult to see that if $t = \lim_{i \to \mcal{U}}\mr{ltp}_{\vphi, \norm{}_i}$ is realized by a structure $(\mf{M}, \norm{}_*)$, then
 \begin{linenomath*} 
 \begin{equation}
 \label{eq:limnorm}
 	\lim_{i \to \mcal{U}}\sup_{\norm{x}_{\ell^1} = 1}\frac{\norm{x}_i}{\norm{x}} = \sup_{\norm{x}_{\ell^1} = 1}\frac{\norm{x}_*}{\norm{x}}
 \end{equation}
 \end{linenomath*} 
 for any structure $(c_{00}, \norm{x}) \in C$. In particular, taking $\norm{}_{\ell^1} = \norm{}$ yields $\lim_{i \to \mcal{U}}\norm{x}_i = \norm{x}_*$ for each $x \in c_{00}$. If in addition 
 \begin{linenomath*} 
 \[\norm{}_1 \leq \norm{}_2 \leq \cdots \leq \norm{}_n \leq \cdots\]
 \end{linenomath*} 
 then $\lim_{i \to \infty}\norm{x}_i = \norm{x}_*$.

 Conversely, if
 \begin{linenomath*} 
 \[\norm{}_1 \leq \norm{}_2 \leq \cdots \leq \norm{}_n \leq \cdots\]
 \end{linenomath*} 
 and also $\lim_{i \to \infty}\norm{x}_i = \norm{x}_*$ for every $x \in c_{00}$, then \eqref{eq:limnorm} holds for any $(c_{00}, \norm{}) \in C$ and any nonprincipal ultrafilter $\mcal{U}$ over $\omega$.

 This motivates the following definitions from \cite{Casazza}:

 \begin{defi}
 	A structure $(c_{00}, \norm{}_{\ell^1}, \norm{}, e_0, e_1, \ldots)$ where $\norm{}$ is an arbitrary norm and $\seq{e}{n}{n<\omega}$ is the standard vector basis of $c_{00}$ is called a \emph{structure based on $c_{00}$}.
 \end{defi}

 \begin{defi}
 \label{def:uniquelydetermined}
 	Let $\mcal{C}$ be a family of structures which are normed spaces based on $c_{00}$, and $\vphi$ a formula for a pair of structures, and $\norm{}_*$ a norm on $c_{00}$.
 	\begin{enumerate}[label=(\alph*)]
 		\item If $\st{\norm{}_i}{i < \omega}$ is a family of norms on $c_{00}$ we say that $\st{\mr{ltp}_{\vphi, \norm{}_i}}{i < \omega}$ \emph{determines $\norm{}_*$ uniquely} if, for every ultrafilter $\mcal{U}$ on $\omega$, the type $t = \lim_{i \to \mcal{U}}\mr{ltp}_{\vphi, \norm{}_i}$ is realized, and $\norm{}_*$ is its unique realization.

 		\item We say that $\norm{}_*$ is \emph{uniquely determined by its $\vphi$-type over $\mcal{C}$} if there is a family of norms $\st{\norm{}_i}{i < \omega}$ on $c_{00}$ in $\mathcal{C}$ such that $\st{\mr{ltp}_{\vphi, \norm{}_i}}{i < \omega}$ determines $\norm{}_*$ uniquely.
 	\end{enumerate}
 \end{defi}

 Notice that if $\norm{}_i$ denotes the $i$-th iterate of the Tsirelson norm, and $\norm{}_T$ the Tsirelson norm, then $\lim_{i \to \infty}\norm{x}_i = \norm{x}_T$ for each $x \in c_{00}$, and so we have the following result from \cite{Casazza}:

 \begin{prop}
 	Let $L$ be a language for pairs of structures, $\mcal{C}$ a class of structures $(c_{00}, \norm{}_{\ell^1}, \norm{})$ such that the norm completion of $(c_{00}, \norm{})$ is a Banach space including $\ell^p$ or $c_{0}$, and let $\vphi(X, Y)$ be the formula defined by \eqref{def:vphi} above. Suppose $\st{(c_{00}, \norm{}_{\ell^1}, \norm{}_i)}{i < \omega}$ is a family of structures in $\mcal{C}$ such that
 	\begin{linenomath*} 
 	\[\norm{}_1 \leq \norm{}_2 \leq \cdots \leq \norm{}_n \leq \cdots\]
 	\end{linenomath*} 
 	and the $\vphi$-type $t = \lim_{i \to \mcal{U}}\mr{ltp}_{\vphi, \norm{}_i}$ is realized by $(c_{00}, \norm{}_{\ell^1}, \norm{}_*)$ in $\mr{Str}(L)$, then $\st{\mr{ltp}_{\vphi, \norm{}_i}}{i < \omega}$ uniquely determines $\norm{}_*$ over $\mcal{C}$. In particular, the Tsirelson norm is uniquely determined by its $\vphi$-type over $\mcal{C}$.
 \end{prop}

 The following result is proved in the same paper \cite{Casazza} and is purely analytical:

 \begin{prop}\label{analysisCI}
 	Let $\norm{}_i$ be the $i$-th iterate in the construction of the Tsirelson norm. Then the following hold:
 	\begin{enumerate}[label=(\roman*)]
 		\item $\sup_{\norm{x}_{\ell^1}=1}\frac{\norm{x}_i}{\norm{x}_j} \leq 1$ for $i < j$.

 		\item $\sup_{\norm{x}_{\ell^1}=1}\frac{\norm{x}_i}{\norm{x}_j} \geq j$ for $i > j$.
 	\end{enumerate}
 	Thus, $\sup_{i<j}\vphi(\norm{}_i, \norm{}_j) \neq \inf_{j<i}\vphi(\norm{}_i, \norm{}_j)$.
 \end{prop}

Proposition \ref{analysisCI}, Theorem \ref{maindlc}, and Remark \ref{rmk:cpctproved} yield:

 \begin{thm}\label{notdef}
 	Let $L$ be a language for pairs of structures, and $\mcal{C}$ a subclass of the class of structures $(c_{00}, \norm{}_{\ell^1}, \norm{})$ such that the norm completion of $(c_{00}, \norm{})$ is a Banach space including some $\ell^p$ or $c_0$, and including the spaces used in the construction of the Tsirelson space. Let $\norm{}_T$ be the Tsirelson norm. Let $\vphi$ be the formula as in \eqref{def:vphi} above. If the space $S^l_\vphi(\mcal{C})$ of $\vphi$-types over $\mcal{C}$ is weakly Grothendieck and satisfies the double ultralimit condition, then $\norm{}_T$ is uniquely determined by its $\vphi$-type over $\mcal{C}$ and that $\vphi$-type is not explicitly definable over $\mcal{C}$.
 \end{thm}

 \begin{rmk*}
 	Notice that here we are dealing with languages for pairs of structures in which only first-order variables are considered. In such a context, definitions like Definition \ref{def:uniquelydetermined} are sound. An interesting logic with countably many formulas is second-order logic. Although it seems as if the topology is confined to subspaces of $2^\omega$, the model theory is not very well developed, even though Chang and Keisler \cite{Chang1973} put the development of the model theory of second-order logic as one of the problems in their classic book. However, see \cite{Vaeaenaenen2023}. It is unclear to the authors whether the results and even the definitions presented here apply to second-order logic, which has been called ``set theory in sheep's clothing'' \cite{Quine1986}. In particular, it would be interesting to know about the definability of Tsirelson's space in second-order arithmetic, which is sometimes conflated with Analysis.
 \end{rmk*}

 We now prove a positive definability result which generalizes the one that appears in \cite{Casazza}. The proof is a variation of the original, in which their use of the Stone-Weierstrass Theorem is replaced by working with the double (ultra)limit condition. This is important because the Stone-Weierstrass Theorem is not nicely extendable beyond compact spaces. The idea of the following theorem is that if the $\varphi$-type of a space $\mf{M}$ is explicitly definable from a well-behaved class based on $c_{00}$, i.e. of spaces including a copy of $c_0$ or $\ell^p$, then $\mf{M}$ is also well-behaved as it includes a copy of $c_0$ or some $\ell^p$.  

  \begin{thm}
 	Let $L$ be a language for pairs of structures, $\vphi$ the formula defined by \eqref{def:vphi} above, and $C$ a subclass of the class of structures based on $c_{00}$ such that every closed subspace of a space in $\clsr{\mcal{C}}$ includes a copy of $c_0$ or $\ell^p$. Assume that the space of $\vphi$-types over $C$ is weakly Grothendieck and satisfies the double ultralimit condition. If the $\vphi$-type of $\mf{M}$ is explicitly definable from $C$, then $\mf{M}$ includes a copy of $c_0$ or some $\ell^p$.
 \end{thm}

 \begin{proof}
 	Suppose that the global $\vphi$-type $\tau$ of $\mf{M}$ is \textit{approximated} by $\seq{\mf{M}}{i}{i < \omega} \subseteq C$ and $\mcal{U}$ is a non-principal ultrafilter on $\omega$, that is
 	\begin{linenomath*} 
 	\[\tau(\lim_{j \to \mcal{V}}\mr{rtp}_{j, \mf{N_j}}) = \lim_{i \to \mcal{U}}\lim_{j \to \mcal{V}}\vphi(\mf{M}_i, \mf{N}_j).\]
 	\end{linenomath*} 
 	We show that there is an $\mf{N} \in \clsr{C}$ such that $\vphi(\mf{M}, \mf{N}) > 0$. Otherwise, $\vphi(\mf{M}, \mf{N}) = 0$ for each $\mf{N} \in \clsr{C}$. Since $\tau$ is explicitly definable, we exchange the order of the ultralimits by taking $\tau(\lim_{j \to \mcal{V}}\mr{rtp}_{j, \mf{N_j}})=\lim_{j \to \mcal{V}}\tau(\mr{rtp}_{j, \mf{N_j}})$ and then obtain
 	\begin{linenomath*} 
 	\[\lim_{i \to \mcal{U}}\lim_{j \to \mcal{V}}\vphi(\mf{M}_i, \mf{N}_j) = \lim_{j \to \mcal{V}}\lim_{i \to \mcal{U}}\vphi(\mf{M}_i, \mf{N}_j) = \lim_{j \to \mcal{V}}\vphi(\mf{M}, \mf{N}_j) = 0.\]
 	\end{linenomath*} 
 	But then consider $\vphi(\mf{M}_i, \mf{M}_j)$ and, by the symmetry of $D(\mf{M}, \mf{N})$, obtain $\vphi(\mf{M}, \mf{M}) = 0$, a contradiction. It is not hard to see that $\vphi(\mf{M}, \mf{N}) > 0$ implies that $D(\mf{M}, \mf{N}) < \infty$, which implies that $e^{\mf{M}}_n \mapsto e^{\mf{N}}_n$ determines an isomorphism (see \cite{Casazza} for the details). This concludes the proof since $\mf{N}$ includes a copy of $c_0$ or $\ell^p$.
 \end{proof}

 Let $L$ be a language for metric structures. The \emph{$\mcal{L}_{\omega_1, \omega}(L)$-formulas} are defined recursively as follows:
 \begin{enumerate}[label=(\alph*)]
 	\item All first-order $L$-formulas are $\mcal{L}_{\omega_1, \omega}(L)$-formulas.

 	\item If $\vphi_1, \ldots, \vphi_n$ are $\mcal{L}_{\omega_1, \omega}(L)$-formulas, and $g:[0, 1]^n \to [0, 1]$ is continuous, then $g(\vphi_1, \ldots, \vphi_n)$ is an $\mcal{L}_{\omega_1, \omega}(L)$-formula.

 	\item If $\st{\vphi_n}{n < \omega}$ is a family of $\mcal{L}_{\omega_1, \omega}(L)$-formulas, then $\inf_n\vphi_n$ and $\sup_n\vphi_n$ are $\mcal{L}_{\omega_1, \omega}(L)$-formulas. These can also be denoted by $\Meet_n\vphi_n$ and $\lJoin_n\vphi_n$ respectively.

 	\item If $\vphi$ is an $\mcal{L}_{\omega_1, \omega}(L)$-formula and $x$ is a variable then $\inf_x\vphi$ and $\sup_x\vphi$ are $\mcal{L}_{\omega_1, \omega}(L)$-formulas.
 \end{enumerate}

 An interesting feature of continuous $\mcal{L}_{\omega_1, \omega}$, noted in \cite{Eagle2015}, is that negation ($\neg$) becomes available in the classical sense: if $L$ is a language and $\vphi$ is an $\mcal{L}_{\omega_1, \omega}(L)$-formula, then we can define $\psi(x) = \lJoin_n\{\vphi(x) + \frac{1}{n}, 1\}$. Then $\mf{M} \ent \psi(a)$ if and only if there is an $n < \omega$ such that $\mf{M} \ent \max\{\vphi(a) + \frac{1}{n}, 1\}$; that is, $\max\{\vphi(a) + \frac{1}{n}, 1\} = 1$, which is the same as $\vphi(a) \leq 1 - \frac{1}{n}$, i.e., $\mf{M} \not\ent \vphi(a)$. Thus, $\mf{M} \ent \psi(x)$ if and only if $\mf{M} \not\ent\vphi(x)$, and so $\psi$ corresponds to $\neg\vphi$.

 It is sometimes useful to restrict one's attention to (countable) fragments of $\mcal{L}_{\omega_1, \omega}(L)$, which are easier to work with than the full logic.

 Let $L$ be a language. A \emph{fragment} $\mcal{F}$ of $\mcal{L}_{\omega_1, \omega}(L)$ is a set of $\mcal{L}_{\omega_1, \omega}(L)$-formulas satisfying:
 \begin{enumerate}[label=(\alph*)]
 	\item Every first-order formula in in $\mcal{F}$.
 	\item $\mcal{F}$ is closed under finitary conjunctions and disjunctions.
 	\item $\mcal{F}$ is closed under $\inf_x$ and $\sup_x$.
 	\item $\mcal{F}$ is closed under subformulas.
 	\item \label{rm:doesnotwork}$\mcal{F}$ is closed under substituting terms for free variables.
 \end{enumerate}
\begin{rmk}\label{fragments}
    Notice that every subset of $\mcal{L}_{\omega_1, \omega}(L)$ generates a fragment, and that every finite set of formulas generates a countable fragment. There are two arguments to support the idea of working with countable fragments of $\mcal{L}_{\omega_1, \omega}(L)$. First, notice that a given proof involves only finitely many formulas, which then generate a countable fragment of $\mcal{L}_{\omega_1, \omega}(L)$. Second (as noted independently by C.~Eagle in a personal communication with the authors), if an object is definable in $\mcal{L}_{\omega_1, \omega}(L)$, then it is definable in a countable fragment by the same argument. This remark does not hold for $\mcal{L}_{\omega_1, \omega_1}$ because of \ref{rm:doesnotwork} above.

\end{rmk}
 
 As the work of Casazza and Iovino deals with continuous logics which are finitary in nature, it was natural to ask whether their result on the undefinability of Tsirelson's space could be proved for continuous $\mcal{L}_{\omega_1, \omega}$, which is arguably a more natural language from the point of view of Banach space theorists.

 In discrete model theory, countable fragments of $\mcal{L}_{\omega_1, \omega}$ have been studied previously; for instance, as we mentioned previously, M.~Morley \cite{Morley1974} showed that the space of types of a countable fragment of $\mcal{L}_{\omega_1, \omega}$ is Polish.



In the continuous case, we note that the space of types of a countable fragment $\mcal{F}$ of continuous $\mcal{L}_{\omega_1, \omega}(L)$ -- including the space of $\varphi$-types over the countable class $\mcal{C}$ of Banach spaces used in the construction of the Tsirelson space that we are interested in -- can be seen as a subspace of $[0, 1]^{\mcal{F}}$ which is metrizable and second countable, and so it is separable and first countable. Hence it is (\groth and) countably tight and thus our results apply. The following results follow from Theorem \ref{notdef} and summarize the preceding discussion:

  \begin{thm}
 	Let $L$ be a language for pairs of structures in a countable fragment of $\mcal{L}_{\omega_1, \omega}$, and $\mcal{C}$ the class of structures $(c_{00}, \norm{}_{\ell^1}, \norm{})$ such that the norm completion of $(c_{00}, \norm{})$ is a Banach space including $\ell^p$ or $c_0$ and used in the construction of the Tsirelson space, and let $\norm{}_T$ be the Tsirelson norm. Let $\vphi$ be the formula as in \eqref{def:vphi} above. Then $\norm{}_T$ is uniquely determined by its $\vphi$-type over $\mcal{C}$ and that $\vphi$-type is not explicitly definable over $\mcal{C}$.
 \end{thm}

Notice that, by Remark \ref{fragments}, if the $\vphi$-type of $\norm{}_T$ was explicitly definable in $\mcal{L}_{\omega_1, \omega}$, it would also be definable in a countable fragment, contradicting the preceding theorem. So we can conclude:

\begin{crl}
    Let $L$ be a language for pairs of structures in $\mcal{L}_{\omega_1, \omega}$, and $\mcal{C}$ the class of structures $(c_{00}, \norm{}_{\ell^1}, \norm{})$ such that the norm completion of $(c_{00}, \norm{})$ is a Banach space including $\ell^p$ or $c_0$ and used in the construction of the Tsirelson space, and let $\norm{}_T$ be the Tsirelson norm. Let $\vphi$ be the formula as in \eqref{def:vphi} above. Then $\norm{}_T$ is uniquely determined by its $\vphi$-type over $\mcal{C}$ and that $\vphi$-type is not explicitly definable over $\mcal{C}$.
\end{crl}

\section{Topological Generalizations}
The hypothesis of \groth's Theorem has been considerably weakened by functional analysts and by topologists. By using the strongest of these results (\cite{Koenig1987,Arhangelskii1997a}), we will considerably generalize \cite{Casazza} and \cite{Hamelb}.

The reader should note that \cite{Koenig1987} calls ``sequential'' the \frech property we defined earlier. General topologists now use ``sequential'' for a weaker property. \cite{Koenig1987} considers topological spaces which are not necessarily regular; we are only interested in completely regular $T_1$ spaces, which simplifies matters. Another simplification is that we may replace their consideration of ``pseudometrizable'' by ``metrizable'', since these are identical for $T_0$ spaces \cite[p.~248]{Engelking1989}. Their notion of ``supersequential'' then amounts to countable tightness plus what \cite{ArhangelskiiFunction} calls ``strongly $\aleph_0$-monolithic'', namely that countable subspaces have metrizable closures. However, their only use of supersequentiality is in a situation where one is asking for a compact subspace to have the property that countable subspaces have metrizable closures. For compact spaces, that countable subspaces have metrizable closures is equivalent to their merely having closures with countable network weight. That property \arhan calls ``$\aleph_0$-monolithic''. That property for $C_p(X)$ is equivalent to $X$ being ``$\aleph_0$-stable''. See \cite{ArhangelskiiFunction} for the definition. Similarly for $C_p(X, [0, 1])$. Observe that a countably tight compact $\aleph_0$-monolithic space is \frech. See \cite[Section II.6]{ArhangelskiiFunction} for further discussion. In particular, the wide class of $\aleph_0$-stable spaces $X$ have that $C_p(X)$ is $\aleph_0$-monolithic. For example, both \lind $\Sigma$ spaces and pseudocompact spaces are $\aleph_0$-stable.

Specializing \cite{Koenig1987} by considering $[0, 1]$ rather than an arbitrary metric space $M$, we consider a topological space $X$, a subspace $T$ of $X$, and a subspace $S$ of $[0, 1]^X$ with the pointwise (i.e., product) topology. K\"onig and Kuhn consider three double limit assertions: \emph{DL I}, \emph{DL II}, \emph{DL III} in strictly descending order of strength. We are particularly interested in DL III because it is closely related to our double limit condition. They then consider five properties (I), \ldots, (V) in strictly descending order of strength which relate relative (countable) compactness of subspaces of $C_p(X, [0, 1])$ to double limit conditions on relatively countably compact subspaces of $X$ and consider under what circumstances these five are equivalent. We shall extract from their paper only the definitions and theorems we need. 
To make it easier for the reader who wants to consult \cite{Koenig1987}, we will use its notation when it is self-explanatory, rather than translating into our slightly different notation. Here is the definition of DL III:

\begin{defi*}
  $S$ and $T$ are in DL III relation if and only if, for each pair of sequences $\{f_p\}_{p \in \nats}$ in $S$ and $\{x_q\}_{q \in \nats}$ in $T$, such that the limits
  \begin{align*}
    &\lim_{q \in \nats}f_p(x_q) =: u_p \in [0, 1] \mbox{ for } p \in \nats,\\
    &\lim_{p \in \nats}f_p(x_q) =: v_q \in [0, 1] \mbox{ for } q \in \nats,\\
    &\lim_{p\in \nats} u_p =: u \in [0, 1], \mbox{ and } \\
    &\lim_{q\in \nats} v_q =: v \in [0, 1]
  \end{align*}
  all exist, we have $u = v$.
\end{defi*}

\begin{thm}
    Suppose $S\subseteq C_p(X,[0,1])$. Then $S$ and $X$ are in DL III relation if and only if $DLC(S,X)$. 
\end{thm}

To see the claimed equivalence, just note that if the double limit exists, so do the other limits. \hfill\qedsymbol

\cite{Koenig1987} mainly uses DL II, but since we are interested in the compact metric space $[0, 1]$, fortunately, their three DL conditions are equivalent. Thus, we do not explicitly state DL II here for it will not be used in any context in which it is not equivalent to DL III. To see this, we quote their 

\textbf{Remark 3.1(iv)} If $S$ and $T$ are in DL III relation and $K = \st{f(x)}{f \in S \mbox{ and } x \in T} \subseteq [0, 1]$ is relatively compact, then $S$ and $T$ are in DL I relation (and hence in DL II relation).

But as their Remark 1.1(b) notes, for $X$ regular, a subset $E$ of $X$ is relatively compact if $\overline{E}$ is compact. $[0, 1]$ is compact, so $K$ is relatively compact. \hfill\qedsymbol

(V) of \cite{Koenig1987} states: \textit{$S$ is in DL II relation with each relatively countably compact $T \subseteq X$. Furthermore, $S(x) =: \st{f(x)}{f \in S} \subseteq [0, 1]$ is relatively compact for all $x \in X$}.

As noted above, the ``Furthermore'' clause is vacuously true in our situation, since $[0,1]$ is compact. Here is

  (I) \textit{$S$ is relatively compact in $C_p(X, [0, 1])$}.

Also relevant is

  (II) \textit{$S$ is relatively countably compact in $C_p(X, [0, 1])$}.

Then (I) implies (II) implies (V).

\cite{Koenig1987} states

\begin{quote}
  \textbf{4.4} \textit{Assume $S \subseteq C_p(X, [0, 1])$ satisfies}

  \begin{quote}
  $(\star)$ If $f \in \overline{S}$, the closure of $S$ in $[0, 1]^X$, is such that $f|_{\overline{T}} \in C_p(\overline{T}, [0, 1])$ for all relatively countably compact $T \subseteq X$, then $f \in C_p(X, [0, 1])$,
  \end{quote}
  \textit{Then (V) implies (I) and hence (I)--(V) are equivalent.}
\end{quote}

\cite{Koenig1987} further remark (p.~160):
\begin{quote}
  One obtains an obvious variant of 4.4 if one formulates both (V) and ($\star$) in that one restricts them to the same subclass $\tau$ of relatively countably compact subsets $T \subseteq X$. In particular, one can take $\tau = \{T\}$. Condition ($\star$) is satisfied trivially if $\overline{T} = X$.
\end{quote}

 It follows immediately that:
 
\begin{thm*}
  If $X$ is countably compact, then (I) and (II) are both equivalent to\\
  (V $'$) $S$ and $X$ are in DL II relation. Furthermore, $S(x) = \st{f(x)}{x \in S} \subseteq [0, 1]$ is relatively compact for all $x \in X$.
\end{thm*} 

But as noted previously, in our situation, the ``Furthermore'' clause is vacuous and $S$ and $X$ being in DL II relation is equivalent to DLC($S,X$), so

\begin{crl*}
  If $X$ is countably compact, then the following are equivalent:\\
  (I*) $S$ is relatively compact in $C_p(X, [0, 1])$.\\
  (II*) $S$ is relatively countably compact in $C_p(X, [0, 1])$.\\
  (V*) DLC($S,X$).
\end{crl*}

Looking back at the undefinablity proofs, what we used about countable tightness and compactness was that they were both weakly \groth and both satisfied the double ultralimit condition. Both of these also hold for countable compactness. Thus what we have proved is what \cite{Casazza} claimed and has since been withdrawn, namely that Tsirelson's space is not definable in any countably compact logic for metric structures. \hfill\qedsymbol

\begin{defi}
    A logic is countably compact if and only if every countable finitely satisfiable collection of formulas is satisfiable.
\end{defi}

Just as in the compact case, equivalent conditions can be formulated for the structure space and the type spaces. Countably compact logics are of interest model-theoretically since a logic is countably compact if and only if it satisfies the \emph{Topological Uniform Metastability Principle} of \cite{Caicedo2019}. See \cite{Eagle2021} for a simple topological proof of this.

There are a variety of generalizations of (countable) compactness that imply weakly \groth---see \cite{Arhangelskii1997a}. Thus one only has to check the double ultralimit condition of type spaces satisfying such generalizations in order to prove a result about undefinability of Tsirelson's space in some logic. This doesn't seem to be a particularly interesting endeavour, so we confine ourselves to two examples: \emph{$k$-spaces} and separable spaces. 

\begin{defi}
  $X$ is a \emph{$k$-space} if for each $Y \subseteq X$, $Y$ is closed if and only if its intersection with every compact subspace of $X$ is closed. 
\end{defi}

\begin{lem}[\cite{Arhangelskii1997a}]
  $k$-spaces are weakly \groth.
\end{lem}

Examples of $k$-spaces include locally compact spaces, \v{C}ech-complete spaces, and first countable spaces. Rather than proving the double limit condition for $k$-spaces, it is sufficient to prove

\begin{lem}
  $k$-spaces satisfy the the double ultralimit condition.
\end{lem}

\begin{proof}
  This will follow from the corresponding result for compact spaces. As in the countable tightness proof, one gets $\{f_n\}_{n<\omega}$ a sequence in $A$ with no limit points in $\overline{A} \cap C_p(X)$. Take a non-principal ultrafilter $\mcal{U}$ over $\omega$ and let $\lim_{n \to \mcal{U}}f_n = g$, where $g$ is discontinuous. It is easy to see that in $k$-spaces, a function is continuous if and only if its restriction to each compact subspace is continuous. Thus there must be a compact $K \subseteq X$ such that $\rstrct{g}{K}$ is discontinuous. Each $\rstrct{f_n}{K}$ is continuous and $\lim_{n \to \mcal{U}} \rstrct{f_n}{K} \to \rstrct{g}{K}$. But now we are back in the compact case.
\end{proof}

  As for separable spaces, by Theorem \ref{thm:densesubspace} and Theorem \ref{thm:grothcntbl}, separable spaces are \groth. It therefore remains to prove

  \begin{lem}
    Separable spaces satisfy the double ultralimit condition.
  \end{lem}

  \begin{proof}
    On page 469 of \cite{Koenig1987}, it is noted that if $X$ is a separable space and $M$ is a metric space, then if $f \in X^M$ is continuous when restricted to $D$, for each countable $D \subseteq X$, then $f$ is continuous. Thus, as in the proof of our Theorem \ref{thm:doublelimit}, we have a sequence $\{f_n\}_{n < \omega}$ and an ultrafilter $\mcal{U}$ such that $\lim_{n\to\mcal{U}}f_n = g$, for some discontinuous $g$. Then there is a countable $Z \subseteq X$ such that $\rstrct{g}{Z}$ is discontinuous. There is then a $y \in \clsr{Z}$ and an $\epsi > 0$ such that $g(y)$ is not in $(g(z) - \epsi, g(z) + \epsi)$ for any $z \in Z$. We now continue as in the proof of Theorem \ref{thm:doublelimit} to conclude that if $X$ is separable, then a subset $A \subseteq C_p(X, [0, 1])$ is relatively compact in $C_p(X, [0, 1])$ if and only if $\udlc(A, X)$. We can then go on to prove a separable version of Theorem \ref{maindlc}.
  \end{proof}

\begin{rmk*}
    Sometimes model theorists and topologists use the same term to mean different things. This can lead to confusion. In \cite{CDI}, the authors' use of \textit{countable compactness} coincides with the usual topological usage; however in \cite{Casazza}, the authors use \textit{countable compactness} to mean f\textit{or each non-principal ultrafilter $\mathcal{U}$ on $\omega$, the ultralimit exists for every sequence}. This is a much stronger property. By a result of J. E. Vaughan \cite{Vaughanhandbook}, this is equivalent to $\omega$-\textit{boundedness}, which asserts that \textit{every countable subset of the space has compact closure}.
\end{rmk*}

\section*{Acknowledgements}
The authors want to express their sincere gratitude to Sasha Lanine and Christopher J. Eagle for their comments, revisions, and suggestions for this paper. In particular, Sasha's detailed study of and suggestions on different versions of our manuscript were an invaluable asset for the improvement of this article. 

\newpage
\bibliographystyle{abbrv}
\bibliography{ms.bib}

\newpage

CLOVIS HAMEL: \hyperlink{chamel@math.toronto.edu}{chamel@math.toronto.edu}\\
Department of Mathematics, University of Toronto, Toronto, ON,\\
Canada

FRANKLIN D. TALL: \hyperlink{f.tall@utoronto.ca}{f.tall@utoronto.ca}\\
Department of Mathematics, University of Toronto, Toronto, ON,\\
Canada

\end{document}